\documentclass[12pt]{article}
\usepackage{t1enc}
\usepackage[latin1]{inputenc}
\usepackage{amsmath,amsthm,amsfonts}
\usepackage{graphicx}
\usepackage{enumerate}
\usepackage{hyperref}
\usepackage{tikz}
\usepackage{caption}
\usepackage{subcaption}

\newtheorem{definition}{Definition}[section]
\newtheorem{thm}[definition]{Theorem}

\newtheorem{lemma}[definition]{Lemma}
\newtheorem{claim}[definition]{Claim}

\newtheorem{prop}[definition]{Proposition}

\newtheorem{conj}[definition]{Conjecture}

\newtheorem{innercustomthm}{Theorem}
\newenvironment{customthm}[1]
{\renewcommand\theinnercustomthm{#1}\innercustomthm}
{\endinnercustomthm}

\newcommand{\AAA}{\mathcal A}
\newcommand{\FF}{\mathcal F}

\newcommand{\RR}{\mathcal R}
\newcommand{\D}{\Delta}
\newcommand{\la}[1]{\operatorname{la}(#1)}

\newcommand{\ceil}[1]{\left\lceil #1 \right\rceil}
\newcommand{\floor}[1]{\left\lfloor #1 \right\rfloor}
\newcommand{\Halfceil}[1]{\left\lceil \frac{#1}{2}\right\rceil}
\newcommand{\halfceil}[1]{\left\lceil #1/2\right\rceil}

\begin{document}
\baselineskip 0.65cm

\title{\bf Linear arboricity of degenerate graphs}
\vspace{6mm}
\author{Guantao Chen\thanks{Partially supported by NSF grant DMS-1855716, gchen@gsu.edu}\,\,, Yanli Hao\thanks{Partially supported by the University Graduate Fellowship of Georgia State University}\,\,, Guoning Yu\footnotemark[2] \thanks{The corresponding author, gyu6@gsu.edu}\\
Department of Mathematics and Statistics, Georgia State University \\
Atlanta, GA 30303, USA}

\date{}
\maketitle

\begin{abstract}
	A {\em linear forest} is a union of vertex-disjoint paths, and the {\em linear arboricity} of a graph $G$, denoted by $\la{G}$, is the minimum number of linear forests needed to partition the edge set of $G$. Clearly, $\la{G} \ge \ceil{\D(G)/2}$ for a graph $G$ with maximum degree $\D(G)$. On the other hand, the  {\it Linear Arboricity Conjecture} due to Akiyama, Exoo, and Harary from 1981 asserts that $\la{G} \leq$ $\ceil{(\Delta(G)+1) / 2}$ for every graph $ G $. This conjecture has been verified for planar graphs and graphs whose maximum degree is at most $ 6 $, or is equal to $ 8 $ or $ 10 $. 
	\vskip .1in
	Given a positive integer $k$,  a graph $G$ is {\it $k$-degenerate} if it can be reduced to a trivial graph by successive removal of vertices with degree at most $ k$. We prove that for any  $k$-degenerate graph $G$, $\la{G} = \ceil{\D(G)/2}$ provided $\D(G) \ge 2k^2 -k$. 
 
	\vskip .2in 
	
\end{abstract}

\section{Introduction}

All graphs in this paper are simple, i.e., finite, undirected, and without loops or multiple edges. A {\em linear forest} is a union of vertex-disjoint paths. Let $G$ be a graph, and $\D(G)$ be the maximum degree of $G$.  The {\em linear arboricity} of  $G$, denoted by $\la{G}$, is the minimum number of linear forests needed to partition its edge set. Clearly, in order to cover all edges incident to a vertex of maximum degree, we need at least $\ceil{\D(G)/2}$ linear forests, i.e., $\la{G} \ge \ceil{\D(G)/2}$. Moreover, it is not difficult to see that $\la{G} \ge \ceil{(\D(G) +1)/2}$ for regular graphs $G$ with even degree. The following conjecture, commonly referred to as the {\em linear arboricity conjecture} (LAC), of Akiyama, Exoo and Harary~\cite{Akiyama1981} in 1981 asserted this bound to be sharp.

\begin{conj}[LAC]
	For every graph $ G $, $ \la{G}\leq\Halfceil{\D(G)+1}$. 
\end{conj}

The LAC implies that  $\ceil{\D(G)/2} \le \la{G} \le \ceil{(\D(G)+1)/2}$ for every graph $G$. An edge coloring of a graph is a partition of its edge set into matchings, which can be considered as a linear forest partition whose each component is a single edge. Hence, the LAC can be viewed as an analog to Vizing's theorem~\cite{Vizing64} that $\D(G) \le \chi'(G) \le \D(G) +1$ for every simple graph $G$, where $\chi'(G)$ is the {\it chromatic index}  of $G$. The conjecture is still wide open,  although there has been some progress in the past nearly 30 years. It was verified for graphs with maximum degree at most $ 6 $, and equal to $ 8 $ or $ 10 $: $\D(G) = 3, 4$ by Akiyama, Exoo, and Harary~\cite{Akiyama1980, Akiyama1981}, $\D(G) = 5, 6, 8$ by Enomoto and P\'eroche~\cite{EP-84}; and $\D(G) =10$ by Guldan~\cite{G-10-86}. the LAC was confirmed for planar graphs by Wu in 1999~\cite{Wu1999planar} and Wu and Wu in 2008~\cite{Wu2008planarD7}. Furthermore, Cygan, Hou, Kowalik, Lu\v{z}ar and Wu in 2011~\cite{conjforplanar} conjectured that the linear arboricity is completely determined for planar graphs provided the maximum degree is large and verified their conjecture for planar graphs with $\D(G)\ge 9$, leaving open only the case $\D(G) =6, 8$. 

\begin{conj}[Cygan et al.] \label{planar-conj}
	For any planar graph $ G $ with $ \D(G)\geq 5 $, $ \la{G}= \Halfceil{\D(G)} $.
\end{conj}

Approximately and asymptotically, Alon in 1988~\cite{alon1988linear} proved that $\la{G} \le \frac{\D(G)}2 + O\left(\frac{\D(G)\log\log\D(G)}{\log\D(G)}\right)$. In the same paper, he also showed that the LAC holds for graphs with girth $\Omega(\D(G))$. Alon and Spencer in 1992 (see~\cite{textprobmethod}) further improved this bound. In 2019, Ferber, Fox and Jain~\cite{Fox2019} further narrowed it to $ \D(G)/2 + \beta \D(G)^{2/3 -\alpha} $, where $\alpha, \beta$ are two positive constants. McDiarmid and Reed~\cite{MR-90} confirmed the LAC for random regular graphs with fixed degrees. Glock, K\"uhn and Osthus~\cite{GKO-16} showed that, for a large range of $p$, a.a.s.\ the random graph $G \sim G_{n,p}$ can be decomposed into $\ceil{\D(G)/2}$ linear forests. Moreover, they also verified the LAC for large and sufficiently dense regular graphs. 

For any positive integer $ k $, a graph $G$ is \textit{$ k $-degenerate} if it can be reduced to a trivial graph by successive removal of vertices of degree at most $ k$; equivalently, every subgraph of $ G $ has a vertex of degree at most $k$. Clearly, trees are $1$-degenerate; outerplanar graphs are $ 2 $-degenerate; and planar graphs are $5$-degenerate.  Moreover, for any surface $ \Pi $, oriented or non-oriented,  there is $ k = k(\Pi) $ such that all graphs embeddable in $ \Pi $ are $ k $-degenerate. Borodin, Kostochka and Woodall~\cite{BORODIN1997184} studied the list edge colorings and list total colorings for degenerate graphs. Zhou, Nakano and Nishizeki~\cite{ZHOU1996598} gave a linear time algorithm for computing the chromatic index of degenerate graphs. Isobe, Zhou and Nishizeki~\cite{Isobe07} showed that the Total Coloring Conjecture holds for $k$-degenerate graphs $G$ if $\D(G) \ge 4k+3$. Inspired by  their result, we show that  the LAC holds for $k$-degenerate graphs with large maximum degrees. More precisely, we completely determine $\la{G}$ for these graphs as follows. 

\begin{thm}\label{thm:main}
	Let $ G $ be a $k$-degenerate graph. If $ \D(G) \ge 2k^2-k $, then $ \la{G} = \Halfceil{\D(G)}$. 
\end{thm}

Basavaraju, Bishnu, Francis and Pattanayak~\cite{3degenerate2020} recently showed that the LAC holds for $3$-degenerate graphs. We also like to make some comments about our proof techniques. Previously, except for planar graphs, all known proofs are based on converting the graphs to regular graphs and use the property that all vertices have the same degree to find a desired linear forest partition. In our proof, we take advantage of the gaps of degrees. More precisely, we reserve some vertices with small degrees as {\it representatives} to give some specific partitions for vertices with large degrees, which in turn gives a linear forest partition step by step for the entire graph. We hope that this new approach may shed some light on the issues that arise when trying to solve for the LAC. We will present some auxiliary lemmas in Section \ref{presection} and prove Theorem~\ref{thm:main} in Section \ref{proofsection}.

\section{Preliminaries}\label{presection}

In this section, we state some basic notation and terminology, present a brute-force method of adding a pair of adjacent vertices with ``low degree'' to a linear forest partition (Lemma~\ref{bruteforce}), and introduce some properties of ``high degree'' vertices in degenerate graphs (Lemma~\ref{Rstar}).

We mainly use the notation from West~\cite{textGTwest}. Let $G$ be a graph with vertex set $V(G)$ and edge set $E(G)$. For a vertex $v\in V(G)$, the {\it neighborhood} of $v$, denoted by $N_G(v)$ or $N(v)$, is the set of vertices adjacent to $v$. The {\it degree} of $v$ in $G$, written $d_G(v)$ or $d(v)$, is the number of edges incident to $v$.  Since we are only considering simple graphs, it follows $d(v) = |N_G(v)|$.  The set of edges that connects $ v $ to a vertex set $ W\subseteq V(G) $ is denoted by $E(v, W)$, that is, $E(v, W)=\{vw\in E(G):\, w\in W\}$. Let $ G \pm v, G \pm e $ and $ G \pm F $ denote the graph obtained from $G$ by adding or deleting a vertex $ v $, an edge $ e $ and an edge set $ F $, respectively.

A partition $ \FF := F_1 \mid F_2 \mid \cdots \mid F_t $ of $E(G)$ is called a \textit{linear forest partition} of $ G $ if the spanning subgraph induced by $F_i$ is a linear forest for each $i\in \{1, \dots, t\}$. For the sake of convenience, we also use $F_i$ to denote the spanning subgraph induced by $F_i$. Notice that graph $F_i$ may contain isolated vertices. Moreover, for any vertex $ v\in V(G) $ and $ F_i\in\FF $, since $F_i$ is a linear forest, we have $ d_{F_i}(v) =0,1 $ or $ 2 $, which corresponds to each of the following three situations, respectively: $ v $ is an isolated vertex, an end-vertex of a path, or an internal vertex of a path in $ F_i $. For each $v\in V(G)$ and $ p \in \{0,1,2\} $, let $ \FF(v,p) = \{F_j\in \FF : \, d_{F_j}(v)=p \} $. By definition, we have $ 2|\FF(v,2)|+|\FF(v,1)|=d(v) $ and $ |\FF(v,2)|+|\FF(v,1)|+|\FF(v,0)| = t $. In our proof, we start with a linear forest partition of a subgraph of $G$ and add the rest edges to it step by step to get a desired linear forest partition of $G$.  At each step, we want the linear forest $\FF$ have a restricted number of the vertices $v$ with $ \FF(v, 2) \ne \emptyset $, where the following brute-force method of adding a pair of adjacent vertices with ``low degree'' plays a critical role. 

\begin{lemma}\label{bruteforce}
	Let $G$ be a graph, and $xy$ be an edge of $G$. If $ 2d(x)+d(y)\leq 2t+2 $, then for any linear forest partition $ \FF= F_1 \mid F_2 \mid \cdots \mid F_t $ of $ G-xy $, there exists an $ i\in $ $ \{1 $, $ \dots $, $ t\} $ such that $\FF^*=F_1\mid \dots \mid F_{i-1}\mid F_i + xy  \mid F_{i+1}\mid \dots \mid F_t$ is also a linear forest partition of $ G $. Moreover, $ \FF^*(v,2) = \FF(v,2) $ for any $ v\in V(G)\setminus\{y\}$.
\end{lemma}

\proof 
Since $\FF =F_1\mid \dots \mid F_t$ partitions $E(G -xy)$, we have $\sum_{i=1}^t d_{F_i}(x)= d(x) -1$ and $\sum_{i=1}^t d_{F_i}(y) = d(y) -1$. Consequently, 
\[
\sum_{i=1}^t (2d_{F_i}(x) + d_{F_i}(y) ) = 2d(x) + d(y) -3 \le  (2t+2) -3 < 2t. 
\]
By average, there exists an $ i\in \{1, \dots, t\} $ such that $2d_{F_i}(x) + d_{F_i}(y) \le 1$. Hence, $d_{F_i}(x) =0$ and $d_{F_i}(y) \le 1$, that is,  $x$ is an isolated vertex in $F_i$ and $y$ is either an isolated vertex or an end-vertex of a path in $F_i$. As a result, $F_i^* := F_i + xy $ is still a linear forest, in which $d_{F_i^*}(x) =1$ and $d_{F_i^*}(y) \le 2$. Note that only the degree of $ x $ and $ y $ change, so that $ \FF^*(v,2) = \FF(v,2)$ for any $ v\ne y$. Therefore, $\FF^* := F_1 \mid \dots \mid F_{i-1} \mid F_i^* \mid F_{i+1} \mid \dots \mid F_t$ is the desired linear forest partition for $G$. \qed

For a given vertex ordering $v_1, v_2, \dots, v_n$ of a graph $ G $, let 
$ L(v_i)  = \{v_1 $, $ v_2 $, $ \ldots $, $ v_{i-1}\} $ and $ R(v_i) = \{v_{i+1}, \dots, v_n\} $ for any vertex $v_i$.  For each $v_i\in V(G)$, let $N_L(v_i) = N(v_i)\cap L(v_i)$, $N_R(v_i) = N(v_i)\cap R(v_i)$. A vertex ordering $v_1, v_2, \dots, v_n$ is called a {\it $k$-degenerate vertex ordering} if $|N_L(v_i)| \le k$ for every $v_i\in V(G)$. It is well-known that 
a graph $G$ is $k$-degenerate if and only if it admits a {\it $k$-degenerate vertex ordering}.   

The following variant of Hall's classic marriage theorem will be applied in our proof. (See \cite[ex.~9]{Diestel2017Chap2}.) Given a family of subsets of a finite set $\AAA = \{ A_1, \dots, A_m\}$, there exist $m$ mutually disjoint subsets $A_i^* \subseteq A_i$ with size $ |A_i^*| \ge r $ for each $i\in \{1, \dots, m\}$ if and only if $ \left\lvert\cup_{i\in I}A_i\right\rvert \ge r |I| $ for every index set $I\subseteq \{1, \dots, m\}$. We call this $m$-tuple $(A_1^*, \dots, A_m^*)$ a {\it system of distinct representatives of size $r$} (or an \textit{$r$-SDR}) of family $\AAA$, and say that $A_i^*$ \textit{represents} $A_i$ for each $i\in \{1, \dots, m\}$.

\begin{lemma}\label{Rstar}
	Let $ G $ be a $ k $-degenerate graph and $ v_1 $, $ \dots $, $ v_n $ be a $ k $-degenerate vertex ordering of $ G $.  For every integer $d$ with $k \le d \le \D(G)$, the family of vertex sets $ \RR: =\{N_R(v_i) :  d(v_i) \ge d \}$ has a $ \floor{\frac{d - k}{k}} $-SDR.
\end{lemma}

\proof 
Let $V_d(G) = \{ v\in V(G): d(v) \ge d\}$. For each $v_i\in V_{d}(G)$, since $d_{G}(v_i) = |N_L(v_i)| + |N_R(v_i)|$ and $|N_L(v_i)| \le k$, we have $ |N_R(v_i)| \ge d(v_i) - k \ge d - k $. For any vertex $v_j$, if $v_j\in N_{R}(v_i)$, then $v_i\in N_L(v_j)$. Since $|N_{L}(v_j)| \le k$, each vertex $v_j\in V(G)$ can only appear in at most $k$ sets $N_{R}(v_i)$ for all $v_i\in V(G)$. Hence, for any $S\subseteq V_{d}(G)$, we have
$$
\left|\bigcup_{v_i\in S} N_{R}(v_i)\right| \ge \frac{1}{k} \sum_{v_i\in S} |N_R(v_i)| \ge \frac{1}{k} (d-k)|S| \ge \floor{\frac{d - k}{k}}|S|.
$$
By Hall's marriage matching theorem, the family $\RR$ has a $\floor{\frac{d - k}{k}}$-SDR with $R^*(v_i)$ for each $v_i\in V_{d}(G)$ satisfying $R^*(v_i) \subseteq N_R(v_i)$ and $|R^*(v_i)| \ge \floor{\frac{d - k}{k}}$. \qed

\section{Proof of Theorem \ref{thm:main}}\label{proofsection}

A graph $ G $ is $1$-degenerate if and only if it is a forest, in which case a linear forest partition of size $ \halfceil{\D(G)} $ can be easily found. We assume $ k > 1 $ in the remainder of our proof. Since the addition or deletion of isolated vertices does not change the linear arboricity, we further assume that the graph being considered does not have isolated vertices.

For any positive integer $d>1$, we call a graph \textit{$(d,1)$-regular} if all its vertices are of degree either $ d $ or $ 1 $. Note that if graph $G$ is $ k $-degenerate with $ \D:=\D(G) $, there is a $(\D,1)$-regular $ k $-degenerate graph $G^*$ containing $G$ as a subgraph. We can construct $G^*$ from $G$ by adding $ \D - d(v) $ new vertices as neighbors for each $ v \in V(G) $ satisfying $ 1<d(v)<\D $. Clearly, $\la{G}\le \la{G^*} $ since $G \subseteq G^*$. With this, it suffices that Theorem \ref{thm:main} holds for $(\D, 1)$-regular $k$-degenerate graphs. We state the following theorem which gives Theorem \ref{thm:main}.

\begin{customthm}{\ref*{thm:main}*}\label{tech:main}
	Let $G$ be a $( \D,1)$-regular $ k $-degenerate graph. If $ \D \ge 2k^2-k $, then there exists a linear forest partition $ \FF = F_1 \mid F_2 \mid  \cdots  \mid F_t $ of $ G $ with $ t = \Halfceil{\D} $.
\end{customthm}

\proof
Suppose $ \D \geq 2k^2-k $. Let $G$ be a $(\D,1)$-regular $ k $-degenerate graph, and $ V_{\D} = \{v:\, d(v) = \D \} $. Then, $ d(v) = 1 $ if $ v\notin V_{\D} $. Let $v_1, \dots, v_n$ be a $k$-degenerate vertex ordering of $G$. Clearly, we may assume that $G$ is connected. Moreover, $i < j$ whenever $d(v_i) =\D$ and $d(v_j) =1$, that is, degree $ 1 $ vertices are placed after degree $ \Delta $ vertices along the vertex ordering. 

Applying Lemma~\ref{Rstar}, we get an $ r $-SDR of the family $ \{ N_R(v_i) : \, v_i\in V_{\D}\} $ where 
$$
r = \floor{\frac{\D - k}{k}} \ge 2k-2 \ge 2,
$$ 
that is, there are mutually disjoint vertex sets $ R^*(v_i)\subseteq N_R(v_i)$ for all $ v_i\in V_{\D} $ such that $ |R^*(v_i)| =r $ for all $ v_i\in V_{\D} $. Additionally, we assume $ R^*(v_i) = \emptyset $ when $ v_i\notin V_{\D} $ for consistency.

For each $i\in \{1, 2, \dots, n\}$, let $ G_i $ be the spanning subgraph of $ G $ induced by edges incident to at least one vertex in $\{v_1$, $v_2$, $\dots$, $v_i\}$. Clearly, $G_1$ is the union of a star centered at $v_1$ and isolated vertices, and $G_1\subseteq G_2 \subseteq \dots \subseteq G_n = G$. The following technical result implies Theorem~\ref{tech:main}. 
Let $t=\Halfceil{\D}$. 

\begin{prop}\label{cla:Gi}
	For each $i \in \{1,\ldots, n\}$, $ G_i $ has a linear forest partition $ \FF^i = F_1^i \mid F_2^i \mid  \cdots  \mid F_t^i $ satisfying that 
	\begin{equation}\label{condi31}
		\FF^i(v,2)=\emptyset \quad \text{ for all } v\in \cup_{j>i} R^*(v_j). \tag{$ * $}
	\end{equation}
\end{prop}

We notice that $ G_{n-1} = G_n = G $, so that $ \FF^{n-1} $ gives the disred partition which completes the proof of Theorem~\ref{tech:main}. We prove Proposition \ref{cla:Gi} by constructing each $ \FF^i $ inductively. The property (\ref{condi31}) will facilitate each step of the construction.

First, for $ G_1 $, by the definition, all its edges are incident to $v_1$ and $d_{G_1}(v_1) = d(v_1) = \D\le 2t$. By pairing edges in the star, we get a linear forest partition $\FF^1 = F_1^1\mid \dots \mid F_t^1$. Since no vertex other than $ v_1 $ can be an internal vertex of a path in any $F_j^1$ for $ j \in \{1,\ldots ,t\}$, it follows that $ \FF^1(v,2)=\emptyset $ for all $ v\ne v_1 $. Hence, (\ref{condi31}) is satisfied.

Suppose that $ i\ge 2 $, and we have constructed a linear forest partition $ \FF^{i-1}=F^{i-1}_1 \mid \cdots \mid F^{i-1}_t $ of $ G_{i-1}$ such that $\FF^{i-1}(v, 2) = \emptyset$ for each $v\in \cup_{j> i-1} R^*(v_j)$. Next, we construct a desired linear forest partition $\FF^i$ based on $\FF^{i-1}$.  

If $ v_i\notin V_{\D} $, then $ d(v_i) = d_{G_i}(v_i) = 1 $. Because $ G $ is connected and $k>1$, no two vertices of degree $ 1 $ in $ G $ can be adjacent. Based on the assumption that every vertex of degree $ 1 $ appears after every vertex of degree $ \D $ in the $ k $-degenerate ordering $ v_1, v_2, \ldots, v_n $ of $G$, the only neighbor $ v_j $ of $ v_i $ must be in $ L(v_i) $, i.e., $ E(v_i,R(v_i))=\emptyset$. Since $ E(G_i)=E(G_{i-1})\cup E(v_i,N_R(v_i))$, it follows that $G_i = G_{i-1}$. Hence, $\FF^i = \FF^{i-1}$ gives a desired partition of $G_{i-1}$. We now assume $ v_i\in V_{\D} $, i.e., $ d(v_i) = \D $.

Let $ W = N_R(v_i)\setminus R^*(v_i)$. We first add edges from $E(v_i, W)$ one by one to $G_{i-1}$. We denote by $ H_\ell $ the graph after adding $\ell$ edges for each $\ell \in \{0,1,2,\ldots, |W|\} $, and by $ H $ the graph $ H_{|W|} $ that we get eventually.
\begin{claim}\label{linearpartitionH}
	For each $ \ell\in \{0,1,2,\ldots, |W|\} $, there exists a linear forest partition $ \FF^{(i-1)\ast \ell} $ of $ H_{\ell} $ such that 
	\begin{equation*}
		\FF^{(i-1)\ast \ell}(v,2)=\emptyset \text{ for all } v\in \cup_{j\ge i} R^*(v_j).
	\end{equation*}
\end{claim}

\begin{proof}
	For $ \ell = 0 $, no edge is to be added. We let $ \FF^{(i-1)\ast 0} = \FF^{i-1} $, so that Claim \ref{linearpartitionH} holds obviously. Suppose $ \ell\ge 1 $, and there is a linear forest partition $ \FF^{(i-1)\ast (\ell-1)} $ of $ H_{\ell-1} $ such that $ \FF^{(i-1)\ast (\ell-1)} (v,2)=\emptyset \text{ for all } v\in \cup_{j\ge i} R^*(v_j) $.
	Since $ H_l\subseteq G_i $, we have $ d_{H_{\ell}} (u) \le d_{G_i} (u) \le k $ for any $u \in W$. Using the equality  $ |W| =|N_R(v_i)|-|R^*(v_i)|= \D -d_{G_{i-1}} (v_i)- |R^*(v_i)|$, we have $ d_{H_{\ell}}(v_i) = d_{G_{i-1}} (v_i)+\ell \le d_{G_{i-1}} (v_i)+ |W| = \D - |R^*(v_i)| = \D - r $. Since $r\ge 2k -2$, we have
	\begin{equation*}\label{brutedeg}
		d_{H_\ell}(v_i) +2 d_{H_\ell}(v_j) \le (\D-r) + 2k \le \D + 2 \le 2t+2.
	\end{equation*}
	Applying Lemma~\ref{bruteforce}, we add an edge $v_iu\in E(v_i, W)$ to $\FF^{(i-1)\ast (\ell-1)}$ such $ \FF^{(i-1)\ast \ell} $ is a linear forest partition of $ H_\ell $ satisfying $\FF^{(i-1)\ast \ell} (v, 2) = \FF^{(i-1)\ast (\ell-1)}(v, 2)$ for every $v\ne v_i$. Since $ \FF^{(i-1)\ast (\ell-1)} (v,2)=\emptyset \text{ for all } v\in \cup_{j\ge i} R^*(v_j) $, we have $ \FF^{(i-1)\ast \ell} (v,2)=\emptyset \text{ for all } v\in \cup_{j\ge i} R^*(v_j) $.
\end{proof}


To avoid cumbersome notation, we denote the linear forest $ \FF^{(i-1)\ast |W|} = F_1^{(i-1)\ast |W|} \mid \cdots \mid F_t^{(i-1)\ast |W|} $ of $ H = G_{i-1} + E(v_i, W) $ by $\FF = F_1\mid \dots \mid F_t$. 
In the remainder of the proof, it is convenient to view the linear forest partition $\FF = F_1\mid \dots \mid F_t$ of graph $H$ as a \textit{linear forest coloring} $\phi_0: E(H) \rightarrow \{1, \dots, t\}$ such that $\phi_0 (e) = j$ if and only if $e\in F_j$ for $ j \in \{1,\ldots,t\} $.  By Claim \ref{linearpartitionH}, there exists a linear forest coloring $ \phi_0 $ of $ H $ satisfying $ \FF(v,2) = \emptyset $ for all $ v\in \cup_{j\ge i}R^*(v_j) $. We call such coloring a \textit{desired linear forest coloring}, and denote the set of all desired linear forest colorings of $ H $ by $ \Phi_0(H) $.

Let $ \phi_0\in \Phi_0(H) $ which generates a desired linear forest partition $\FF = F_1\mid \dots \mid F_t$. For each $p\in\{0,1,2\}$, we denote by $ C_p(v) = \{j\mid d_{F_j}(v)=p\} $ which is the set of colors that appear $ p $ times at the edges incident to $ v $. In order to color edges in $E(G_i) \setminus E(H) = E(v_i, R^*(v_i))$, we first assign colors arbitrarily from $C_0(v_i)\cup C_1(v_i)$.  We call a coloring $ \phi $ of edges in $ E(G_i) $ a \textit{$ \phi_0 $-extension} if $ \phi(e) = \phi_0(e) $ for any edge $ e\in E(H) $ and $ \phi(e)\in C_0(v_i)\cup C_1(v_i) $ for any edge $ e\in E(G_i)\setminus E(H) = E(v_i, R^*(v_i)) $. A $ \phi_0 $-extension $ \phi $ is called \textit{$ v_i $-saturated} if for each color $ j\in\{1,\ldots, t\} $, there are at most two edges with color $ j $ incident to $ v_i $. In other words, for edges in $ E(v_i, R^*(v_i)) $, $ \phi $ is $ v_i $-saturated each color in $ C_0(v_i) $ is used at most twice and each color in $ C_1(v_i) $ is used at most once. For a coloring $ \phi_0\in \Phi_0(H) $, since $ C_0(v_i) $, $ C_1(v_i) $ and $ C_2(v_i) $ give a partition of $ \{1,\ldots, t\} $, we have $ |C_2(v_i)|+|C_1(v_i)|+|C_0(v_i)| = t $. And since $ \phi_0 $ is a coloring of edges in $ E(H) $, we have $2|C_2(v_i)|+|C_1(v_i)|=d_{H}(v_i) =d(v_i)-r = \D -r $. Hence, $ 2|C_0(v_i)| + |C_1(v_i)| = 2t - (\D - r) $.  Applying the fact that $ \D\le 2t=2\left\lceil\frac{\D}{2}\right\rceil \le \D + 1$, we have
\begin{equation}\label{color10ff}
	r\le 2|C_0(v_i)| + |C_1(v_i)| \le r+1.
\end{equation}
Combining the lower bound of (\ref{color10ff}) and $ r =|E(v_i, R^*(v_i))| $, we see that there exists a coloring of $ E(v_i, R^*(v_i)) $ such that each color in $ C_0(v_i) $ is used at most twice and each color in $ C_1(v_i) $ is used at most once. This is to say that there exists a  $ v_i $-saturated $ \phi_0 $-extension. Further, we call a coloring $ \phi $ a \textit{feasible} coloring of $ E(G_i) $ if for any $ v\in V(G) $ and $ j\in \{1,\ldots, t\} $, there are at most two edges of color $ j $ incident to $ v $.

\begin{claim}\label{feasi}
	Let $ \phi $ be a $ \phi_0 $-extension for some $ \phi_0\in \Phi_0(H) $. Then, $ \phi $ is $ v_i $-saturated if and only if it is feasible.
\end{claim}
\begin{proof}
	The necessity is clear. We prove the sufficiency. Suppose that $ \phi $ is $ v_i $-saturated. Let $\FF = F_1\mid \dots \mid F_t$ be the desired linear forest partition generated by $ \phi_0 $, and let $ \FF' = F_1'\mid \dots \mid F_t' $ be the corresponding partition of $ E(G_i) $ generated by $ \phi $. We consider any vertex $ v \in V(G) $. It suffices if $ d_{F'_j}(v)\leq 2 $ for all $ j\in \{1,\ldots, t\} $. If $ v\notin \{v_i\}\cup R^*(v_i) $, we have $ d_{F'_j}(v) = d_{F_j}(v) \le 2 $ for all $ j\in \{1,\ldots, t\} $. Also if $ v \in R^*(v_i) $, then $ d_{F_j}(v)\le 1 $ since $ \FF(v,2) = \emptyset $, which gives $ d_{F'_j}(v) \le d_{F_j}(v)+1 \le 2 $ for all $ j\in \{1,\ldots, t\} $. Now suppose that $ v=v_i $. Since $ \phi $ is a $ \phi_0 $-extension which colors the edges in $ E(v_i,R^*(v_i)) $ only using colors in $ C_0(v_i)\cup C_1(v_i) $, we have $ d_{F'_j}(v) = d_{F_j}(v) \le 2 $ for any $ j\in C_2(v_i) $. Moreover, given that $ \phi  $ is $ v_i $-saturated, each color in $ C_0(v_i) $ is used at most twice and each color in $ C_1(v_i) $ is used at most once. Therefore, if $ j\in C_0(v_i) $, then $ d_{F'_j}(v) \le d_{F_j}(v)+2 = 0+2 = 2 $. And if $ j\in C_1(v_i) $, then $ d_{F'_j}(v) \le d_{F_j}(v)+1 = 1+1 = 2 $.
\end{proof}

Note that a feasible coloring of a graph generates a linear forest partition of $ G_i $ if and only if there is no monochromatic cycle.
For any $ v_i $-saturated $ \phi_0 $-extension $ \phi $, we let $B(\phi) = \{ v\in R^*(v_i) \mid v_iv $ is in a cycle of $ F'_j $ for some $ j\in \{1,\ldots,t\} \}$. Clearly, $\phi$ gives a linear forest partition $\FF' = F_1'\mid \dots \mid F_t'$ of $ G_i $ provided $B(\phi) = \emptyset$. Moreover, since $ R^*(v_j)\cap R^*(v_i) = \emptyset $ for all $ j>i $ and $ \phi $ is a $ \phi_0 $-extension which keeps the coloring of every edge $ e\notin E(v_i, R^*(v_i)) $, it follows that $ \FF'(v,2) = \FF(v,2) = \emptyset $ for all $ v\in \cup_{j>i}R^*(v_j) $. Thus additionally, the linear forest partition $ \FF' $ of $ G_i $ satisfies the property (\ref{condi31}), which gives $ \FF^i $ in Proposition \ref{cla:Gi}. The following claim implies that such a linear forest partition $ \FF' $ exists.


\begin{claim}\label{case12}
	There exists a coloring $ \phi_0\in \Phi_0 $ and a $ v_i $-saturated $ \phi_0 $-extension $ \phi $ such that $ B(\phi) = \emptyset $.
\end{claim}

\begin{figure}[h]
	\centering
	\resizebox{0.7\textwidth}{!}{%
		\begin{tikzpicture}
			\begin{scope}
				\filldraw [line width=4mm,join=round,black!10]
				(0,1) rectangle (0.8,3.7) 
				(3.5,0) rectangle (4.7,5.9);
				\node[above] at (0.4,3.9) {$ N_L(v_i) $};
				\node[above] at (4.1,6.1) {$ N_R(v_i) $};
				
				\filldraw [line width=4mm,join=round,black!20]
				(3.7,0.2) rectangle (4.5,2.5)
				(3.7,3.5) rectangle (4.5,5.4);
				
				\node[above,font=\small] at (4.1,2.6) {$ R^*(v_i) $};
				\node[above,font=\small] at (4.1,5.5) {$ W $};
				\node[above,font=\small] (uu) at (4.1,0.5) {$ w $};
				
				\coordinate (v) at (2,2.8);
				\node[above] at (2,2.8) {$ v_i $};
				
				\draw[thick] 
				(v) -- (0.4,3.3)
				(v) -- (0.4,2) node[midway,below]{$ \xi $}
				(v) -- (4.1,5)
				(v) -- (4.1,4)
				(v) -- (4.1,2.2)
				(v) -- (4.1,1.3);
				\draw[thick,dotted] (v) -- (4.1,0.5);
				
				\fill (2,2.8) circle (2pt);
				
				\draw[thick,out=-60,in=-170] (0.4,2) edge node[below]{$ \xi $} (4.1,0.5);
				\node at (2,-0.7) {\parbox{0.3\linewidth}{\subcaption{$ \xi\in\FF(v_i,1) $}\label{subfig:a}}};
			\end{scope}
			
			\begin{scope}[xshift=7cm]
				\filldraw [line width=4mm,join=round,black!10]
				(0,1) rectangle (0.8,3.7) 
				(3.5,0) rectangle (4.7,5.9);
				\node[above] at (0.4,3.9) {$ N_L(v_i) $};
				\node[above] at (4.1,6.1) {$ N_R(v_i) $};
				
				\filldraw [line width=4mm,join=round,black!20]
				(3.7,0.2) rectangle (4.5,2.5)
				(3.7,3.5) rectangle (4.5,5.4);
				
				\node[above,font=\small] at (4.1,2.6) {$ R^*(v_i) $};
				\node[above,font=\small] at (4.1,5.5) {$ W $};
				\node[above,font=\small] (u) at (4.1,2.2) {$ w $};
				\node[above,font=\small] (uu) at (4.1,1.3) {$ w' $};
				
				\coordinate (v) at (2,2.8);
				\node[above] at (2,2.8) {$ v_i $};
				
				\draw[thick] 
				(v) -- (0.4,3.3)
				(v) -- (0.4,2)
				(v) -- (4.1,5)
				(v) -- (4.1,4)
				(v) -- (4.1,0.5);
				\draw[thick,dotted] 
				(v) -- (4.1,2.2)
				(v) -- (4.1,1.3);
				
				\fill (2,2.8) circle (2pt);
				
				\draw[thick,out=-10,in=95] (4.1,2.2) edge (5.1,1.6);
				\draw[thick,out=-95,in=-5] (5.1,1.6) edge (4.1,1.3);
				\node[right] at (5.1,1.6) {$ \xi $};
				\node at (2,-0.7) {\parbox{0.3\linewidth}{\subcaption{$ \xi\in\FF(v_i,0) $}\label{subfig:b}}};
			\end{scope}		
		\end{tikzpicture}
	}
	\caption{Two possible monochromatic cycles in $ \FF' $}
	\label{fig}
\end{figure}
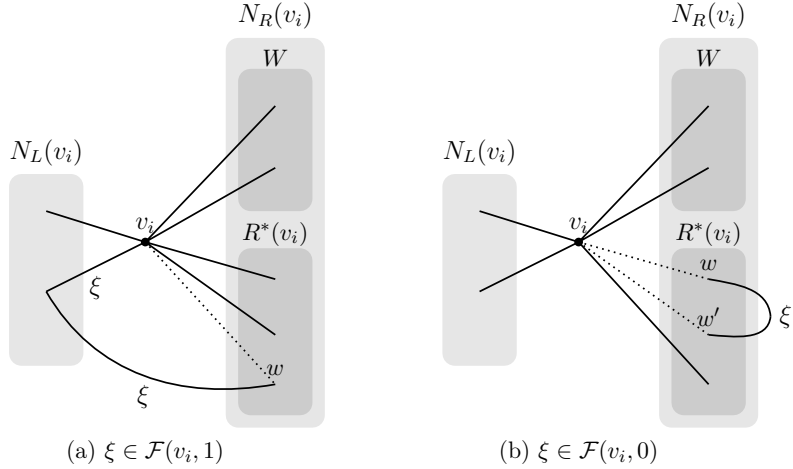

\begin{proof}
Suppose on the contrary that $ B(\phi) \ne \emptyset $ for any $ \phi_0\in \Phi_0 $ and any $ v_i $-saturated $ \phi_0 $-extension $ \phi $. Let $ \phi_0\in \Phi_0 $ and $\phi$ be a $ v_i $-saturated $ \phi_0 $-extension such that $|B(\phi)| $ is the minimum. By Claim \ref{feasi},  $ \phi $ is feasible. Let $ \FF = F_1\mid \dots \mid F_t $ and $ \FF' = F_1'\mid \dots \mid F_t' $ be the corresponding partitions obtained from $ \phi_0 $ and $ \phi $, respectively. Then, there exist $w\in B(\phi)$ and $\xi = \phi(v_i w)$ such that  $F_{\xi}'$ has a cycle $ C $ containing edge $v_iw$. Let $ P = C-v_i $. Clearly, $ P $ is a path in $ F_\xi $. We claim that $|V(P)\cap (R^*(v_i)\setminus \{w\})| \le 1$. 
For a vertex $u\in R^*(v_i)$, we note that $ d_{F_\xi}(u)\le 1$
since $ \FF(u,2) = \emptyset $. Hence, any vertex in $ V(P)\cap R^*(v_i) $ is an end-vertex of $ P $. It follows that $ |V(P)\cap R^*(v_i)|\le 2 $. Since $ w $ is obviously an end-vertex of $ P $, there exists at most one vertex $ w'\in V(P)\cap (R^*(v_i)\setminus \{w\}) $, in other words, $|V(P)\cap (R^*(v_i)\setminus \{w\})| \le 1$. There are only two possible types of cycles in $ F'_{\xi} $ as illustrated in Figure~\ref{fig}: if $ F_{\xi}\in\FF(v_i,1) $, $ P $ contains no vertices in $ R^*(v_i) $ other than $ w $ (see figure~\ref{subfig:a}); and if $ F_{\xi}\in \FF(v_i,0) $, the equality $|V(P)\cap (R^*(v_i)\setminus \{w\})| = 1$ holds. In the latter case, we denote the unique vertex in the intersection by $w'$ (see figure~\ref{subfig:b}). Clearly, $\phi(v_iw') = \phi(v_iw) =\xi$.  

Note that in either case above, we have $ F_{\xi} \in \FF(w,1) $. To complete the proof, we prove that there exists a $ v_i $-saturated $ \phi_0 $-extension $ \phi^* $ obtained from $ \phi $, by finding a color $ \eta $ such that replacing the color $ \xi $ on $ v_i w $ by $ \eta $ does not produce any new monochromatic cycles. We consider three cases as follows. 

{\bf Case 1.} $\FF(w, 0)\setminus \FF(v_i, 2) \ne \emptyset$.

Let $F_{\eta}\in \FF(w, 0)\setminus \FF(v_i, 2)$. If there is no edge $ v_i u\in E(v_i, R^*(v_i)) $ such that $ \phi(v_i u) = \eta $, let $ \phi^* $ be obtained from $ \phi $ by assigning $ v_i w $ the color $ \eta $ to replace $ \xi $. If there exists $ v_i u\in E(v_i, R^*(v_i)) $ such that $ \phi(v_i u) = \eta $, let $\phi^*$ be obtained from $\phi$ by swapping colors on $v_iw$ and $v_iu$, i.e., $\phi^*(v_iw) = \eta$ and $\phi^*(v_iu) = \xi$.  We denote by $\mathcal{F}^{*}=F_{1}^{*}\mid \cdots \mid F_{t}^{*}$ the partition given by the coloring $\phi^{*}$. Clearly, $ \phi^* $ is $ \phi_0 $-saturated since no edge color in $ H $ has been changed. Moreover, $ \phi^* $ is $ v_i $-saturated since neither color $ \xi $ or $ \eta $ has been used more than twice.


Further, we have $F_{\eta}^{*} \in \mathcal{F}^{*}(w, 1)$, which implies that there is no monochromatic cycle containing the edge $v_{i} w$ in $\phi^{*}$. Thus $w \notin B\left(\phi^{*}\right)$. Now if the edge $v_{i} u$ does not exist, then we have $\left|B\left(\phi^{*}\right)\right| \leq|B(\phi)|-1$, which contradicts our choice of $\phi$. So, let us assume that the edge $v_{i} u$ exists. 
Suppose that there is a monochromatic cycle $ C^{*} $ containing the edge $v_{i} u$ in $\phi^{*}$ with color $ \xi $. Let $ P^{*} = C^{*} - v_i u $. Note that $ u\notin\left\{w, w^{\prime}\right\} $ since $\phi\left(v_{i} w\right) = \phi\left(v_{i} w^{\prime}\right)=\xi \neq \eta=\phi\left(v_{i} u\right)$. By our earlier observation that the only vertices in $C$ that belong to $R^{*}\left(v_{i}\right)$ are $w$ and $w^{\prime}$ (if exists), it follows that $ P^* $ is not a subpath of $ C $. Since $ v_i $ belongs to both $ C $ and $ P^* $ in $ F'_{\xi} $, there exists a vertex $ v\in C\cup P^* $ such that $d_{F'_{\xi}}(v) = 3$. However, this contradicts to the fact that $ \phi $ is feasible. Therefore we can conclude that there is no monochromatic cycle containing the edge $v_{i} u$ in $\phi^{*}$. Then we again get $\left|B\left(\phi^{*}\right)\right| \leq|B(\phi)|-1$, which contradicts the choice of $\phi$.


{\bf Case 2.} $\FF(v_i, 0)\setminus \{F_\xi\} \ne \emptyset$.

Let $F_{\eta} \in \mathcal{F}\left(v_{i}, 0\right) \backslash\left\{F_{\xi}\right\}$. We choose a vertex $u \in R^{*}\left(v_{i}\right)$ such that $\phi\left(v_{i} u\right)=\eta$ as follows. If there exists a path from $v_{i}$ to $w$ in $F_{\eta}^{\prime}$, then as $F_{\eta} \in \mathcal{F}\left(v_{i}, 0\right)$, the vertex adjacent to $v_{i}$ on this path belongs to $R^{*}\left(v_{i}\right)$ and we choose that vertex as $u$. If there is no path from $v_{i}$ to $w$ in $F_{\eta}^{\prime}$, then we choose $u$ to be any vertex in $R^{*}\left(v_{i}\right)$ such that $\phi\left(v_{i} u\right)=\eta$. Recall that in (\ref{color10ff}) we have $2|C_0(v_i)| + |C_1(v_i)| \le r+1$. Since $ |R^*(v_i)| = r $, all colors in $ C_0(v_i) $ have been used under the coloring $ \phi $. Such a vertex $ u $ always exists as $\eta \in C_0(v_i)$. We now obtain a new coloring $\phi^{*}$ from $\phi$ by swapping colors on $v_iw$ and $v_iu$, i.e., $\phi^*(v_iw) = \eta$ and $\phi^*(v_iu) = \xi$. We denote by $\mathcal{F}^{*}=F_{1}^{*}\left|F_{2}^{*}\right| \cdots \mid F_{t}^{*}$ the partition given by the coloring $\phi^{*}$. As we have just exchanged the colors on two edges in $E(v_{i},R^*(v_i))$, the coloring $ \phi^* $ remains to be a $ v_i $-saturated $ \phi_0 $-extension.

Suppose that there exists a monochromatic cycle containing the edge $v_{i} w$ in $\phi^{*}$. Then, there is a path $P^{\prime}$ from $v_{i}$ to $w$ in $F_{\eta}^{*}-v_{i} w$. As $\phi^{*}\left(v_{i} u\right) \neq \eta$, we know that the edge $v_{i} u$ does not belong to $P^{\prime}$. Thus, $P^{\prime}$ is also a path in $F_{\eta}^{\prime}$ from $v_{i}$ to $w$. Note there exists at most one such path $ P' $ since $ d_{F'_\eta}(w) = d_{F'_\eta}(v_i) = 1 $ and $ d_{F'_\eta}(v) \le 2 $ for all $ v\in V(G) $. However, by our choice of $u$, if such $ P' $ exists, the edge $v_{i} u$ belongs to $P^{\prime}$, which is a contradiction. We can therefore conclude that there is no monochromatic cycle containing the edge $v_{i} w$ in $\phi^{*}$.

Next, suppose that there is a monochromatic cycle containing the edge $v_{i} u$ in $\phi^{*}$. Then, there is a path $P^{\prime \prime}$ from $v_{i}$ to $u$ in $F_{\xi}^{*}-v_{i} u$. Clearly, this path does not contain $v_{i} w$, and therefore $P^{\prime \prime}$ is also a path from $v_{i}$ to $u$ in $F_{\xi}^{\prime}-v_{i} w$. Then, as before, $P^{\prime \prime}$ is a subpath of $P$. Note that $u \neq w^{\prime}$ since $\phi\left(v_{i} w^{\prime}\right)=\xi \neq \eta=\phi\left(v_{i} u\right) .$ However, this implies that $P$ contains the vertex $u \in R^{*}\left(v_{i}\right)$ that is different from $w^{\prime}$ and $w$, which is a contradiction. Therefore, there is no monochromatic cycle containing the edge $v_{i} u$ in $\phi^{*}$. We then have $\left|B\left(\phi^{*}\right)\right| \leq|B(\phi)|-1$, contradicting the choice of $\phi$.

{\bf Case 3.} $\FF(w, 0)\subseteq \FF(v_i, 2)$ and $\FF(v_i, 0)\setminus \{F_\xi\} = \emptyset$.

We claim $ k=2 $ in this case. Since $\mathcal{F}(w, 0) \subseteq \mathcal{F}\left(v_{i}, 2\right)$, it follows that $\mathcal{F}\left(v_{i}, 0\right) \cup \mathcal{F}\left(v_{i}, 1\right) \subseteq \mathcal{F}(w, 1)$, which gives $\left|\mathcal{F}\left(v_{i}, 0\right)\right|+\left|\mathcal{F}\left(v_{i}, 1\right)\right| \leq$ $|\mathcal{F}(w, 1)| \leq d_{H}(w) \leq k-1$. Then, $2\left|\mathcal{F}\left(v_{i}, 0\right)\right|+2\left|\mathcal{F}\left(v_{i}, 1\right)\right| \leq 2 k-2 \leq r $. By (\ref{color10ff}), we have $2\left|\mathcal{F}\left(v_{i}, 0\right)\right|+\left|\mathcal{F}\left(v_{i}, 1\right)\right| = 2 |C_0(v_i)| + |C_1(v_i)| \geq r$. Combining these two inequalities, we have $\left|\mathcal{F}\left(v_{i}, 1\right)\right|=0$, and $2\left|\mathcal{F}\left(v_{i}, 0\right)\right| = r \geq 2 k-2$, which implies $\left|\mathcal{F}\left(v_{i}, 0\right)\right| \geq k-1$. Since we also have $\left|\mathcal{F}\left(v_{i}, 0\right)\right| \leq$ $|\mathcal{F}(w, 1)| \leq k-1$, we now have $\left|\mathcal{F}\left(v_{i}, 0\right)\right|=|\mathcal{F}(w, 1)|=k-1$ and $r=2 k-2$. As $\FF(v_i, 0)\setminus \{F_\xi\} = \emptyset$, we have $\left|\mathcal{F}\left(v_{i}, 0\right)\right| \leq 1$. Therefore, we get $k-1 \leq 1$, or in other words, $k \leq 2$. Since we have assumed $k \geq 2$, we now have $k=2$.

The fact that $k=2$ further implies $\left|\mathcal{F}\left(v_{i}, 0\right)\right|=|\mathcal{F}(w, 1)|=k-1=1$ and $r = 2k-2=2$. We then have $\mathcal{F}\left(v_{i}, 0\right)=\{\xi\}$, which in turn gives $R^{*}\left(v_{i}\right)=\left\{w, w^{\prime}\right\}$. Note that $ |N_L(v_i)|\le k $. As $\Delta \geq 2 k^{2}-k \geq 6$, we know that $\left|N_{R}\left(v_{i}\right)\right|=d(v_i)-|N_L(v_i)|\ge\Delta-k \geq 4$, which implies that $|W| \geq 2$. Let $u$ be any vertex in $W$ and let $\eta=\phi\left(v_{i} u\right)$. Clearly, $\eta \neq \xi$. Let $\phi^{*}$ be the coloring obtained from $\phi$ by swapping colors on $v_iw$ and $v_iu$, i.e., $\phi^*(v_iw) = \eta$ and $\phi^*(v_iu) = \xi$. We denote by $\mathcal{F}^{*}=F_{1}^{*}\left|F_{2}^{*}\right| \cdots \mid F_{t}^{*}$ the partition given by the coloring $\phi^{*}$. Since we changed the coloring of $ H $, we let $ \phi_0' $ be the new coloring of $ H $ in which $ \phi_0'(v_iu) = \xi $ and $ \phi_0'(v_iv) = \phi_0(v_iv) $ for the neighbors $ v\ne u $.
As we have just exchanged the colors of two edges incident on $v_{i}$, we have $d_{F_{j}^{*}}\left(v_{i}\right) \leq 2$ for each $j \in\{1,2, \ldots, t\} $. Therefore, $ \phi^* $ is $ v_i $-saturated.

Since $d_{G_i}(w) =2$ and $v_iw$ is contained in a cycle in $F_{\xi}'$, the other edge incident to $w$ in $G_i$ is also given with color $\xi$.  Thus, there is no monochromatic cycle containing the edge $v_{i} w$ in $\phi^{*}$ as $\phi^*(v_iw) = \eta$. Suppose that there is a monochromatic cycle containing the edge $v_{i} u$ in $\phi^{*}$. Then as before, there is a path $P^{\prime}$ from $v_{i}$ to $u$ in $F_{\xi}^{*}-v_{i} u$. Note that $v_{i} w$ is not in $P^{\prime}$ as $\phi^{*}\left(v_{i} w\right) \neq \xi$. Then $P^{\prime}$ is  a path from $v_{i}$ to $u$ in $F_{\xi}^{\prime}-v_{i} w$. Since $v_{i}$ belongs to $P^{\prime}$, we can then conclude that $P^{\prime}$ is a subpath of $C$. Since the degree of $u$ in $ G_i $ is at most $ 2 $, and one edge incident on it is colored $\eta$ in $\phi$, we can conclude that $u$ is an end-vertex of $P$, which contradicts the fact that the end-vertices of $P$ are $v_{i}$ and $w$. Therefore, there is no monochromatic cycle containing the edge $v_{i} u$ in $\phi^{*}$. Then we have $\left|B\left(\phi^{*}\right)\right| \leq|B(\phi)|-1$, contradicting our choice of $\phi$. This completes our proof.
\end{proof}

\section{Remark}
In addition, along a similar approach as the proof of Theorem \ref{thm:main}, the original upper bound stated in the LAC holds for a slightly smaller maximum degree than in Theorem \ref{thm:main}. We state it as the following theorem.

\begin{thm}\label{thm:add}
	Let $ G $ be a $k$-degenerate graph. If $ \D(G) \ge 2k^2-2k $, then $ \la{G} \le \Halfceil{\D(G)+1}$. 
\end{thm}

Similarly, now we let $ t = \halfceil{\D+1} $, and state the following equivalent theorem in terms of $(\D,1)$-regular $ k $-degenerate graphs. 

\begin{customthm}{\ref*{thm:add}*}\label{tech:add}
	Let $G$ be a $(\D,1)$-regular $ k $-degenerate graph. If $ \D \ge 2k^2-2k $, then there exists a linear forest partition $ \FF = F_1 \mid F_2 \mid  \cdots  \mid F_t $ with $ t = \Halfceil{\D+1} $.
\end{customthm}

\proof
There are some slight differences from the proof of Section \ref{proofsection} in terms of the computation related to $ t $. Suppose $ \D \ge 2k^2 - 2k $. Then the size of each representative of the family $ \{ N_R(v_i) : \, v_i\in V_{\D}\} $ satisfies $ r \ge 2k-3 $. It is readily seen that all conditions related to $ \D, r $ and $ t $ still hold. 

Therefore, the same construction works. We can find all desired linear forest partitions $ \FF^i = F_1^i\mid F_2^i\mid \cdots \mid F_t^i $ for any $G_i$, $i = 1,\ldots, n$ such that $ \FF^i(v,2)=\emptyset $ for all $ v\in \cup_{j>i} R^*(v_j) $.
\qed

We strongly believe the quadratic lower bound $2k^2 -k$ for the maximum degree in Theorem~\ref{thm:main} may be reduced to $2k$ since it is only required in our proof of Lemma~\ref{Rstar}. We wonder if there is a way to reduce it to a linear bound in terms of $k$.

\section{Acknowledgments}

Thanks to Mathew Francis for his detailed review report. His report not only corrected typos and suggested on grammars, but also helped us a lot on making our paper more readable and understandable.

\bibliographystyle{plain}

\bibliography{LinearArboricity.bib}

\end{document}